\documentclass[12pt]{article}
\usepackage{amsmath,amsthm,amssymb}
\usepackage{hyperref}

\usepackage{mathrsfs}
\allowdisplaybreaks
\newcommand{\rom}[1]{{\rm #1}}

\newtheorem{theorem}{Theorem}[section]

\newtheorem{lemma}[theorem]{Lemma}

\theoremstyle{definition}

\newtheorem{example}[theorem]{Example}
\newtheorem{remark}[theorem]{Remark}

\newcommand{\X}{\mathfrak X}
\newcommand{\scrC}{\mathscr C}
\newcommand{\wt}{\widetilde}
\newcommand{\supp}{\operatorname{supp}}

\begin{document}

\begin{center}{\Large \bf
Equilibrium  Kawasaki dynamics and determinantal point processes}\end{center}

{\large Eugene Lytvynov}\\ Department of Mathematics,
Swansea University, Singleton Park, Swansea SA2 8PP, U.K.\\
e-mail: \texttt{e.lytvynov@swansea.ac.uk}\vspace{2mm}

{\large Grigori Olshanski}\\
 Institute for Information Transmission Problems, Bolshoy Karetny 19, Moscow
127994, GSPÐ4, Russia; Independent University of Moscow, Russia;
Department of Mathematics, Higher School of Economics, Moscow, Russia.\\
 e-mail: \texttt{olsh2007@gmail.com}\vspace{2mm}

 \vspace{1.5mm}

 \vspace{1.5mm}

{\small

\begin{center}
{\bf Abstract}
\end{center}

\noindent  Let $\mu$ be a point process on a countable discrete space
$\mathfrak X$. Under assumption that $\mu$ is quasi-invariant with respect to
any finitary permutation of $\mathfrak X$, we describe a general scheme for
constructing an equilibrium  Kawasaki dynamics for which $\mu$ is a
symmetrizing (and hence invariant) measure. We also exhibit a two-parameter
family of point processes $\mu$ possessing  the needed quasi-invariance
property. Each process of this family is determinantal, and  its correlation
kernel is the kernel of a projection operator in $\ell^2(\mathfrak X)$.
\vspace{2mm}

\section{Introduction}

\subsection{Determinantal point processes}

Let $\mathfrak X$ be a locally compact topological space and let $\mathcal
B(\mathfrak X)$ be the Borel $\sigma$-algebra on $\mathfrak X$. The {\it
configuration space\/} $\Gamma:=\Gamma_{\mathfrak X}$ over $\mathfrak X$ is
defined as the set of all subsets $\gamma\subset\mathfrak X$  which are locally
finite. Such subsets are called {\it configurations\/}. The space $\Gamma$ can
be endowed with the vague topology, i.e., the weakest topology on $\Gamma$ with
respect to which all maps $\Gamma\ni\gamma\mapsto \sum_{x\in\gamma}f(x)$, $f\in
C_0(\mathfrak X)$, are continuous. Here $C_0(\mathfrak X)$ is the space of all
continuous real-valued functions on $\mathfrak X$ with compact support. We will
denote by $\mathcal B(\Gamma)$ the Borel $\sigma$-algebra on $\Gamma$. A
probability measure $\mu$ on $(\Gamma,\mathcal B(\Gamma))$ is called a {\it
point process\/} on $\mathfrak X$. For more detail, see, e.g., \cite{Len},
\cite{Sosh}.

A point process $\mu$ can be described with the help of {\it correlation
functions\/}. Let $m$ be a reference Radon measure on $(\mathfrak X,\mathcal
B(\mathfrak X))$. The $n$th correlation function ($n=1,2,\dots$) is a
non-negative measurable symmetric function $k_\mu^{(n)}(x_1,\dots,x_n)$ on
$\mathfrak X^n$ such that, for any measurable symmetric function
$f^{(n)}:\mathfrak X^n\to[0,\infty]$, one has
\begin{align}
&\int_\Gamma \sum_{\{x_1,\dots,x_n\}\subset\gamma}
f^{(n)}(x_1,\dots,x_n)\,\mu(d\gamma)\notag\\&\qquad
 =\frac1{n!}\,
\int_{\mathfrak X^n} f^{(n)}(x_1,\dots,x_n)
k_\mu^{(n)}(x_1,\dots,x_n)\,m(dx_1)\dotsm m(dx_n)\label{1A}
\end{align}
Under a mild condition on the growth of correlation functions as $n\to\infty$,
they determine the point process uniquely \cite{Len}.

A point process $\mu$ is called {\it determinantal\/} if there exists a
function $K(x,y)$ on $\mathfrak X^2$, called the {\it correlation kernel\/},
such that
$$
k_\mu^{(n)}(x_1,\dots,x_n)=\operatorname{det}[K(x_i,x_j)]_{i,j=1}^n,\quad
n=1,2,\dots,
$$
see e.g.\ \cite{Sosh}, \cite{Bor}. Assume that $K(x,y)$ is the  integral kernel
of a selfadjoint, locally trace class operator $K$  in the (real or complex)
space $L^2(\mathfrak X,m)$.  Then, by \cite{Sosh}, the corresponding
determinantal point process exists  if and only if  $\mathbf 0\le K\le\mathbf
1$. (Note, however, that there are natural examples of determinantal point
processes whose correlation kernel $K(x,y)$ is non-Hermitian, see, e.g.,
\cite{BO-1}, \cite{Bor}.)

If we additionally assume that $K<\mathbf 1$, i.e., $1$ does not belong to the
spectrum of $K$, then, as shown in \cite{GeYoo}, the corresponding
determinantal point process $\mu$ is  Gibbsian in a weak sense.   More
precisely,  there exists a measurable function $r:\mathfrak
X\times\Gamma\to[0,+\infty]$ such that
 \begin{equation}
\int_\Gamma \mu(d\gamma)
\sum_{x\in\gamma}
F(x,\gamma) =\int_\Gamma \mu(d\gamma)\int_{\mathfrak X}
m(dx)\, r(x,\gamma) F(x,\gamma\cup x)\label{mecke}
\end{equation}
for any measurable function $F:\mathfrak X\times \Gamma \to[0,+\infty].$  Here
and below, for simplicity of notation, we just write $x$ instead of $\{x\}$.
Note that, in the theory of point processes,  \eqref{mecke}
 is called {\it condition\/} $\Sigma_m'$, see \cite{MWM}.

It should be, however, emphasized that, in most applications, the selfadjoint
operator $K$ appears to be an orthogonal projection in $L^2(\mathfrak X,m)$,
which is why the condition $K<\mathbf 1$ is not satisfied.

\subsection{Kawasaki dynamics}

Informally, by a {\it Kawasaki dynamics\/} we mean a continuous time Markov
process on $\Gamma$ in which ``particles'' occupying positions $x\in\gamma$
randomly hop over the space $\mathfrak X$. Such a dynamics should be described
by the rate $c(\gamma,x,y)$ at which a particle occupying position $x$ of
configuration $\gamma$ jumps to a new position $y$. We will be interested in
{\it equilibrium dynamics\/}, which means that the process admits a
symmetrizing (and hence invariant) measure $\mu$, and we want to consider the
time-reversible evolution preserving $\mu$.

In the statistical mechanics literature one usually takes as $\X$ the lattice
$\mathbb Z^d$, $x$ and $y$ are assumed to be neighboring sites of the lattice,
and $\mu$ is a Gibbs measure.

Using the theory of Dirichlet forms, Kondratiev {\it et al.\/} \cite{KLR}
constructed an equilibrium Kawasaki dynamics with a continuous space $\X$ and a
classical double-potential Gibbs measure of Ruelle type as the symmetrizing
measure $\mu$. This approach was extended in \cite{LO} to the case when $\mu$
is  a determinantal point process. However, since the authors of \cite{LO}
heavily used formula \eqref{mecke}, their construction of the Kawasaki dynamics
was restricted to the case of a selfadjoint operator $K$ with $K<\mathbf 1$.

Let us also note that, in    \cite{ShYoo} (in the case where $\mathfrak X$ is a
discrete space) and in \cite{LO}, an equilibrium Glauber dynamics  (i.e., a
spacial birth-and-death process) was constructed which has a determinantal
point process as symmetrizing measure. To this end, one again needed that
$K<\mathbf 1$. Under the same assumption, an equilibrium diffusion process for
a determinantal measure was constructed in \cite{Yoo}.

The purpose of the present note is to describe a general scheme for
constructing an equilibrium Kawasaki dynamics in the case of a discrete space
$\X$. The crucial property of a measure $\mu$ on $\Gamma$ which makes it
possible to construct the dynamics, is the quasi-invariance of $\mu$ with respect
to finitary permutations of $\X$. We show that the construction can be applied
to a family of determinantal measures $\mu$ whose correlation kernels $K$ are
projection kernels. Thus, at least in a concrete case we can remove the
undesirable restriction $K<\mathbf 1$.

More precisely, we will deal with the {\it gamma kernel measures\/} which were
introduced and studied by Borodin and Olshanski in \cite{BO}. As shown there,
these determinantal point processes arise from several models of
representation-theoretic origin through certain limit transitions. The
quasi-invariance property of the gamma kernel measures is established in
\cite{Olsh2}. It would be interesting to find other natural examples of
discrete determinantal point processes possessing the quasi-invariance
property.

It is worth noting that although, on abstract level, one can find some
similarity between determinantal point processes and Gibbs measures, the Gibbs
measure technique seems to be hardly applicable to determinantal measures. The
main reason is that, in determinantal point processes, the interaction between
``particles'' is non-local. Note also that for lattice spin Gibbs measures (at
least in the case of their uniqueness), the needed quasi-invariance property is
obvious from the very definition, which is not the case for determinantal
measures.

In the present note we employ the Dirichlet form approach, but, with the
exception of a reference to the nontrivial abstract existence theorem for Hunt
processes associated with regular Dirichlet forms \cite{Fu}, we manage with
fairly easy and standard arguments.

\section{Discrete point processes}\label{Sect2}
From now on, we will assume that $\mathfrak X$ is a countable set with discrete
topology.  Thus, a configuration in $\mathfrak X$ is an arbitrary subset of
$\mathfrak X$.  We can therefore identify $\Gamma$ with $\{0,1\}^{\mathfrak
X}$, so that a subset $\gamma$ of $\mathfrak X$ is identified with its
indicator function.  Then the vague topology on $\Gamma$ is nothing else but
the product topology on $\{0,1\}^{\mathfrak X}$. Thus, $\Gamma$ is a  compact
topological space.

Let $m$ be the counting measure on $\mathfrak X$:  $m(\{x\})=1$ for each
$x\in\mathfrak X$. Let $\mu$ be a point process on $\mathfrak X$. Then, by
\eqref{1A},
$$
k_\mu^{(n)}(x_1,\dots,x_n)=\mu(\gamma\in\Gamma:
\{x_1,\dots,x_n\}\subset\gamma)
$$
for distinct  points $x_1,\dots,x_n\in\mathfrak X$, otherwise
$k_\mu^{(n)}(x_1,\dots,x_n)=0$. In this situation, the correlation functions
uniquely identify the corresponding point process. If $\mu$ is determinantal,
then its correlation kernel is simply a matrix with the row and columns indexed
by points of $\X$.

A permutation $\sigma:\mathfrak X\to\mathfrak X$ is said to be {\it finitary\/}
if it fixes all but finitely many points in $\mathfrak X$. The simplest example
of a nontrivial finitary permutation is the transposition $\sigma_{x,y}$, where
$x,y$ are distinct points in $\mathfrak X$; by definition $\sigma_{x,y}$
permutes $x$ and $y$ and leaves invariant all other points. The finitary
permutations form a countable group, which we  denote as $\mathfrak S(\X)$. The
transpositions $\sigma_{x,y}$ constitute a set of generators for $\mathfrak
S(\X)$. The tautological action of the group $\mathfrak S(\mathfrak X)$ on
$\mathfrak X$ gives rise to a natural action of this group on the space
$\Gamma=\{0, 1\}^{\mathfrak X}$ by homeomorphisms:
$$
(\sigma(\gamma))(x):=\gamma(\sigma^{-1}(x)),\quad \sigma\in\mathfrak
S(\mathfrak X),\ \gamma\in\Gamma,\ x\in\mathfrak X.
$$
Therefore, $\mathfrak S(\X)$ also acts on the set of all probability measures
on $\Gamma$. A measure $\mu$ on $\Gamma$ is said to be {\it quasi-invariant\/}
with respect to the action of $\mathfrak S(\mathfrak X)$ if for any element
$\sigma\in \mathfrak S(\mathfrak X) $ the measure $\mu$ is equivalent to
$\sigma(\mu)$. As easily seen, it suffices to require that, for any
transposition $\sigma_{x,y}$, the measure $\sigma_{x,y}(\mu)$ is absolutely
continuous with respect to $\mu$. (Note that, since $\sigma_{x,y}^2$ is the
identity, the latter condition implies that  the measure $\sigma_{x,y}(\mu)$ is
equivalent to $\mu$.)

If $\mu$ is a determinantal point process with correlation kernel $K(x,y)$,
then, for any  $\sigma\in\mathfrak S(\mathfrak X)$, the measure $\sigma(\mu)$
is also determinantal, with correlation kernel
$K^\sigma(x,y)=K(\sigma^{-1}(x),\sigma^{-1}(y))$. However, in the general case,
it is not clear how to decide whether $\mu$ is equivalent to $\sigma(\mu)$ by
looking at the kernels $K$ and $K^\sigma$.

One can raise a more general question \cite{Olsh2}: Given two determinantal
measures on $\{0,1\}^{\mathfrak X}$, how to test their equivalence (or, on the
contrary, disjointness) by inspection of their correlation kernels? Note that a
product measure $\mu$ on $\{0,1\}^{\mathfrak X}$ is the determinantal point
process with the correlation kernel $K(x,y)$ given by
$$
K(x,y)=\begin{cases}\mu(\gamma: \gamma(x)=1),&\text{if }x=y,\\
0,&\text{otherwise}.\end{cases}
$$
For product measures, the answer to the above question is well known: it is given
by the classical Kakutani theorem \cite{Kak}.

\section{Gamma kernel measures}
The quickest way of introducing the gamma kernel is as follows (see
\cite{Olsh1}). We say that a couple $(z,z')$ of complex numbers is {\it
admissible\/} if
    \begin{equation}\label{3A}
    (z+n)(z'+n)>0\quad \text{for all }n\in\mathbb Z.
    \end{equation}
 This condition is satisfied if
    \begin{align*}&\text{either $z\in\mathbb C\setminus\mathbb Z$ and $z'=\bar z$}\\
    &\text{or there exists $m\in\mathbb Z$ such that $m<z$, $z'<m+1$}.
    \end{align*}
In what follows, we fix an admissible couple  of parameters, $(z,z')$.

Next, we identify $\mathfrak X$ with the lattice $\mathbb Z':=\mathbb
Z+\frac12$ of half-integers and consider the following second-order difference
operation on the lattice $\mathbb Z'$:
\begin{align*}
\mathcal (\mathcal D_{z,z'}f)(x)&=\sqrt{\left(z+x+\frac12\right)
\left(z'+x+\frac12\right)}f(x+1)-(2x+z+z')f(x)\\
&\quad+\sqrt{\left(z+x-\frac12\right)\left(z'+x-\frac12\right)}f(x-1),
\end{align*}
where $x\in\mathbb Z'$ and $f(x)$ is a test function on $\mathbb Z'$. Note that
$x\pm\frac12$ is an integer for any $x\in\mathbb Z'$. Consequently, by virtue
of  \eqref{3A}, the quantities under the sign of square root are strictly
positive, so that we may extract the positive square root.

Let $D_{z,z'}$ stand for the operator in $\ell^2(\mathbb Z')$ which is defined
by the operation $\mathcal D_{z,z'}$ on the domain consisting of all
$f\in\ell^2(\mathbb Z')$ such that $\mathcal D_{z,z'}f\in\ell^2(\mathbb Z')$.
One can prove that $D_{z,z'}$ is selfadjoint and has simple, purely continuous
spectrum filling the whole real axis.

Let $K_{z,z'}$ be the spectral projection associated with the selfadjoint
operator $D_{z,z'}$ and corresponding to the positive part of the spectrum.
That is, denoting by $Q(\cdot)$ the projection-valued measure on $\mathbb R$
that governs the spectral decomposition of $D_{z,z'}$, we set
$K_{z,z'}:=Q((0,+\infty))$. We define $\mu_{z,z'}$ as the determinantal measure
on $\Gamma$ with the correlation kernel $K_{z,z'}(x,y)$---the integral kernel
of the operator $K_{z,z'}$.

As shown in \cite{BO}, the kernel $K_{z,z'}(x,y)$ admits an explicit expression
in terms of the classical $\Gamma$-function:
$$
K_{z,z'}(x,y)=\frac{\sin(\pi z)\sin(\pi
z')}{\pi\sin(\pi(z-z'))}\cdot\frac{A(x)B(y)-B(x)A(y)}{x-y}\, ,\quad
x,y\in\mathbb Z',
$$
where
$$
A(x)=\frac{\Gamma(z+x+\frac12)}{\sqrt{\Gamma(z+x+\frac12)\Gamma(z'+x+\frac12)}}\,, \quad B(x)=
\frac{\Gamma(z'+x+\frac12)}{\sqrt{\Gamma(z+x+\frac12)\Gamma(z'+x+\frac12)}}\,.
$$
Note that the quantity under the sign of square root is always strictly
positive. The above expression is well defined provided that $x\ne y$, $z\ne
z'$.
For $x=y$, one takes the formal limit as $y\to x$, which leads to
$$
K_{z,z'}(x,x)=\frac{\sin(\pi z)\sin(\pi z')}{\pi\sin(\pi(z-z'))}
\left(\psi\left(z+x+\frac12\right)-\psi\left(z'+x+\frac12\right)\right),\quad x\in\mathbb Z',
$$
where $\psi(x)=\Gamma'(x)/\Gamma(x)$ is the logarithmic derivative of the gamma
function. The definition for the case $z=z'\in\mathbb R\setminus\mathbb Z$ is
obtained by taking the limits as $z'\to z$.

We call $K_{z,z'}(x,y)$ and $\mu_{z,z'}$ the {\it gamma kernel\/} (with
parameters $z,z'$) and the  {\it gamma kernel measure\/}, respectively. For
more detail about the gamma kernel measures and related measures on partitions
(the so-called {\it z-measures\/}), see \cite{BO-1,BO, BO1, BO2}.

\begin{theorem}[\cite{Olsh2}]\label{th3A}
All gamma kernel measures $\mu_{z,z'}$ are quasi-invariant with respect to the
action of the group $\mathfrak S(\mathbb Z')$.
 \end{theorem}

As shown in \cite{Olsh2}, the Radon--Nikod\'ym derivative of
$\sigma_{x,y}(\mu_{z,z'})$ relative to $\mu_{z,z'}$ admits an explicit
expression. This expression involves an infinite product which converges only
when $\gamma$ belongs to a relatively meager subset of the whole configuration
space $\Gamma$. Fortunately, this subset has full measure.

\section{Construction of dynamics}

Let again $\mathfrak X$ be as in Section~\ref{Sect2}. Let $\mu$ be a point
process on $\mathfrak X$ which is  quasi-invariant with respect to the action
of $\mathfrak S(\mathfrak X)$. Thus,
\begin{equation}\label{4A}
\text{for any distinct $x,y\in\mathfrak X$, the measures $\mu$ and
$\sigma_{x,y}(\mu)$ are equivalent}
\end{equation}

Let $\mathscr C$ denote the space of all cylinder functions on $\Gamma$, i.e.,
a function $F:\Gamma\to\mathbb R$ is in $\mathscr C$ if and only if there
exists a finite subset $\Lambda\subset\mathfrak X$ and a function $\wt
F:\{0,1\}^\Lambda\to\mathbb R$ such that $F(\gamma)=\wt F(\gamma_{\Lambda})$,
$\gamma\in \Gamma$, where $\gamma_\Lambda$ is the restriction of $\gamma$ to
$\Lambda$.  Note that each $F\in\mathscr C$ is continuous on $\Gamma$. Let
$\tilde\scrC$ stand for the dense subspace in $L^2(\Gamma,\mu)$ formed by the
images of the cylinder functions. If $\operatorname{supp}\mu=\Gamma$, then
$\tilde\scrC$ can be identified with $\scrC$. (Here and below $\supp\mu$, the
topological support of $\mu$, is the smallest closed subset of full measure.)

Let $\tilde{\mathfrak X}^2:=\{(x,y)\in\mathfrak X^2\mid x\ne y\}$. Let
$c:\Gamma\times \tilde{\mathfrak X}^2\to[0,\infty)$  be a measurable function
satisfying the symmetry relation $c(\gamma,x,y)=c(\gamma,y,x)$. That is, given
$\gamma$, $c(\gamma,x,y)$ depends on the {\it unordered\/} couple $\{x,y\}$.
(Here and below all relations involving $c(\gamma,x,y)$ are assumed to hold for
$\mu$-a.a. $\gamma\in\Gamma$.) As will be clear from the formulas below, we
will actually exploit only the restriction of the function $c$ to the subset of
those triples $(\gamma,x,y)$ for which $\gamma$ contains precisely one of the
points $x,y$; then, informally, $c(\gamma,x,y)$ is the rate of jump from
position $\gamma\cap\{x,y\}$ to the new position
$\{x,y\}\setminus(\gamma\cap\{x,y\})$. For this reason one can call
$c(\gamma,x,y)$ the {\it rate function\/}.

For any $(x,y)\in\tilde\X^2$ we define the operator $\nabla_{x,y}$ acting on
functions $F(\gamma)$ according to formula
$$
(\nabla_{x,y}F)(\gamma):=F(\sigma_{x,y}(\gamma))-F(\gamma).
$$

In accordance with the intuitive meaning of the rate function, we would like to
define the generator of the future dynamics by the formula
\begin{equation}\label{4D}
-(AF)(\gamma):=\sum_{(x,y)\in\tilde\X^2}c(\gamma,x,y)(\nabla_{x,y}F)(\gamma),
\end{equation}
where $F$ ranges over an appropriate space of functions on $\Gamma$. (We put
the minus sign in the left-hand side for convenience, because we want $A$
to be a nonnegative operator.) To make the definition rigorous we have to
specify the domain of the operator, and we also have to impose suitable
conditions on the rate function. Let us consider the following two conditions:
\begin{gather}
\mbox{\rm ``Symmetry'': for any fixed $(x,y)\in\tilde\X^2$, the measure
$c(\gamma,x,y)\mu(d\gamma)$ is $\sigma_{x,y}$-invariant.} \label{4B}\\
\mbox{\rm ``$L^2$-condition'': for any fixed $x\in\mathfrak X$,} \quad
\sum_{y\in\mathfrak X,\ y\ne x}c(\,\cdot\,,x,y)\in L^2(\Gamma,\mu). \label{4C}
\end{gather}
The ``symmetry condition'' is analogous to the ``detailed balance condition''
for lattice spin systems of statistical mechanics; such a condition is
necessary if we want the future Markov process to be symmetric (that is,
reversible with respect to $\mu$). The ``$L^2$-condition'' is a technical
assumption; below we will also introduce a weaker condition, see \eqref{4F}.

\begin{lemma}\label{le4A}
Assume the rate function $c(\gamma,x,y)$ satisfies \eqref{4B} and \eqref{4C}.
Then the formula \eqref{4D} correctly determines a nonnegative symmetric
operator $A$ in $L^2(\Gamma,\mu)$ with dense domain $\tilde\scrC$.
\end{lemma}

\begin{proof} Let $F$ range over $\scrC$. If $F=0$ $\mu$-a.e., then, due to the quasi-invariance of $\mu$,
the same holds for $\nabla_{x,y}F$, so that the right-hand side of \eqref{4D}
also equals 0 $\mu$-a.e. Thus, $AF$ depends only on the image of $F$ in
$\tilde\scrC$.

Next, because of \eqref{4C}, $AF\in L^2(\Gamma,\mu)$. Indeed, write
$F(\gamma)=\wt F(\gamma_\Lambda)$ as above, with an appropriate finite subset
$\Lambda\subset\X$. Then $\nabla_{x,y}F$ vanishes when both $x$ and $y$ are
outside $\Lambda$. Therefore, we may assume that at least one of the points
$x,y$ (say, $x$) is inside $\Lambda$. Since $|(\nabla_{x,y}F)(\gamma)|$ is
bounded from above by a constant not depending on $x$, $y$ and $\gamma$, we see from \eqref{4C}
that, for any fixed $x\in\Lambda$, the sum over $y$ of the functions
$c(\,\cdot\,,x,y)(\nabla_{x,y}F)(\cdot)$ is in $L^2(\Gamma,\mu)$. This is sufficient, because there
are only finitely many $x\in\Lambda$.

Let us set \begin{equation}\label{4E} \mathcal
E(F,G):=\frac12\int_{\Gamma}\mu(d\gamma)\sum_{(x,y)\in \tilde{\mathfrak X}^2}
c(\gamma,x,y)(\nabla_{x,y}F)(\gamma)(\nabla_{x,y}G)(\gamma),\quad
F,G\in\tilde{\mathscr C}.
\end{equation}
 Using \eqref{4B},
it is easy to check that
$$
(AF,G)=\mathcal E(F,G), \quad F,G\in\tilde{\scrC},
$$
where $(\,\cdot\,,\,\cdot\,)$ denotes the inner product in $L^2(\Gamma,\mu)$.

Finally, the fact that $A$ is symmetric and nonnegative is evident, because the
bilinear form \eqref{4E} clearly possesses these properties.
\end{proof}

Let us examine the expression \eqref{4E}. Observe that it still makes sense and
correctly defines a symmetric bilinear form on $\tilde\scrC\times\tilde\scrC$
if we replace \eqref{4C} by the weaker condition
\begin{equation}\label{4F}
\mbox{\rm ``$L^1$-condition'': for each $x\in\mathfrak X$, }
\sum_{y\in\mathfrak X,\ y\ne x}c(\,\cdot\,,x,y)\in L^1(\Gamma,\mu);
\end{equation}
the proof is the same as above. As for the symmetry condition \eqref{4D}, it is
actually not restrictive: one can always modify the rate function, without
changing $\mathcal E(F,G)$, in such a way that \eqref{4D} will be fulfilled:
For each $(x,y)$, we simply take the average of the measure $c(\,\cdot\,,x,y)\mu$ and its image under
 $\sigma_{x,y}$ (the resulting measure will remain absolutely
continuous with respect to $\mu$).

\begin{lemma}
Under the ``$L^1$-condition'' \eqref{4F}, the form $(\mathcal E,\tilde\scrC)$
defined by \eqref{4E} is closable on $L^2(\Gamma,\mu)$\rom.
\end{lemma}

\begin{proof}
Note that under the stronger ``$L^2$-condition'' \eqref{4C}, the claim is
evident, because the quadratic form corresponding to a symmetric operator is
always closable. Without \eqref{4C}, the argument is slightly lengthier (cf.
\cite[Example 1.2.4]{Fu}).

For any $F\in\mathscr C$, we abbreviate $\mathcal E(F):=\mathcal E(F,F)$. Let
$(F_n)_{n=1}^\infty$ be a sequence in $\mathscr C$ such that
$\|F_n\|_{L^2(\Gamma,\mu)}\to0$ as $n\to\infty$ and
\begin{equation}\label{4G}
\mathcal E(F_k-F_m)\to0\quad\text{as }k,m\to\infty.
\end{equation}
To prove the closability of $\mathcal E$, it suffices to show that there exists
a subsequence $(F_{n_k})_{k=1}^\infty$ such that  $\mathcal E(F_{n_k})\to0$ as
$k\to\infty$. Since $\|F_n\|_{L^2(\Gamma,\mu)}\to0$ as $n\to\infty$, there
exists a subsequence $(F_{n_k})_{k=1}^\infty$  such  that $F_{n_k}(\gamma)\to0
$ as $k\to\infty$ for $\mu$-a.a.\ $\gamma\in\Gamma$. Then, by \eqref {4A}, for
any  $(x,y)\in\tilde{\mathfrak  X}^2$, $F_{n_k}(\sigma_{x,y}\gamma)\to0 $ as
$k\to\infty$ for $\mu$-a.a.\ $\gamma\in\Gamma$. Therefore, for any
$(x,y)\in\tilde{\mathfrak X}^2$,

\begin{equation}\label{4H}
(F_{n_k}(\sigma_{x,y}\gamma)-F_{n_k}(\gamma))\to0\quad\text{as
}k\to\infty\quad\text{for $\mu$-a.a.\ $\gamma\in\Gamma$}.
\end{equation}
Now, by \eqref{4H} and Fatou's lemma,
\begin{align*}
&2\mathcal E(F_{n_k})=\sum_{(x,y)\in\tilde{\mathfrak
X}^2}\int_{\Gamma}c(\gamma,x,y)(F_{n_k}(\sigma_{x,y}(\gamma))-F_{n_k}(\gamma))^2\mu(d\gamma)
\\ &=\sum_{(x,y)\in\tilde{\mathfrak
X}^2}\int_{\Gamma}c(\gamma,x,y)\big((F_{n_k}(\sigma_{x,y}(\gamma))-F_{n_k}(\gamma))
-\lim_{m\to\infty}(F_{n_m}(\sigma_{x,y}(\gamma))-F_{n_m}(\gamma))
\big)^2\mu(d\gamma)\\
&\le\liminf_{m\to\infty}\sum_{(x,y)\in\tilde{\mathfrak X}^2}\int_{\Gamma}
c(\gamma,x,y)\big((F_{n_k}(\sigma_{x,y}(\gamma))-F_{n_k}(\gamma))
-(F_{n_m}(\sigma_{x,y}(\gamma))-F_{n_m}(\gamma))
\big)^2\mu(d\gamma)\\
&=2\liminf_{m\to\infty}\mathcal E(F_{n_k}-F_{n_m}),
\end{align*}
which by \eqref{4G} can be made arbitrarily small for $k$ large enough.
\end{proof}

We denote by $(\bar{\mathcal E},\mathbb D(\bar{\mathcal E}))$ the closure of
$(\mathcal E,\tilde\scrC)$ on $L^2(\Gamma,\mu)$ (thus $\mathbb D(\bar{\mathcal
E})$ is the domain of $\bar{\mathcal E}$). For the notions of a Dirichlet form
and of a regular Dirichlet form, appearing in the following lemma, we refer to
e.g. \cite[Section~1.1]{Fu}.

\begin{lemma}\label{le4C} Assume \eqref{4F}. Then the form $(\bar{\mathcal
E},\mathbb D(\bar{\mathcal E}))$ just defined is a regular Dirichlet form.
\end{lemma}

\begin{proof}
For each $F\in\mathscr C$, we clearly have $(0\vee F)\wedge 1\in\mathscr C$ and
$\mathcal E((0\vee F)\wedge 1)\le \mathcal E(F)$. Therefore, $(\mathcal
E,\tilde\scrC)$ is Markovian. Since this property is preserved under closing
(see \cite[Theorem 3.1.1]{Fu}), the form $(\bar{\mathcal E},\mathbb
D(\bar{\mathcal E}))$ is Markovian, too. Hence it is a Dirichlet form. Finally,
by the very construction, it is  regular, because $\mathbb D(\bar{\mathcal
E})$ includes $\tilde\scrC$, which is dense in the space of continuous
functions on the compact space $\supp\mu\subseteq\Gamma$.
\end{proof}

\begin{theorem}\label{th4D}
Let, as above, $\mu$ be a $\mathfrak S(\X)$-quasi-invariant probability
measure on the configuration space $\Gamma=\Gamma_\X$ and
$c(\gamma,x,y)=c(\gamma,y,x)$ be a nonnegative measurable function satisfying
\eqref{4B} and \eqref{4F}. Then the corresponding form $(\bar{\mathcal E},\mathbb D(\bar{\mathcal E}))$,
as defined above, gives rise to a conservative symmetric Markov semigroup
$\{T_t\}_{t\ge0}$ in $L^2(\Gamma,\mu)$, which in turn determines a symmetric
Hunt process on $\supp\mu\subseteq\Gamma$.
\end{theorem}

\begin{proof}
The existence of $\{T_t\}$ follows from the fact that $(\bar{\mathcal
E},\mathbb D(\bar{\mathcal E}))$ is a Dirichlet form \cite[Section 1.3]{Fu}.
Conservativity holds because ${\mathcal E}(1)=0$. The existence of a Hunt
process is a consequence of the regularity of the form, see \cite[Chapter
7]{Fu}.
\end{proof}

\begin{remark} If the rate function satisfies the
``$L^2$-condition'' \eqref{4C}, then one can say more. Let $\widehat A$ stand
for infinitesimal generator of the semigroup $\{T_t\}$, so that $-\widehat A$
is the nonnegative selfadjoint operator associated with the form
$(\bar{\mathcal E},\mathbb D(\bar{\mathcal E}))$. Then, by virtue of Lemma
\ref{le4A}, $\widehat A$ is the Friedrichs' extension of the symmetric operator
$A$ determined by \eqref{4D}. In the general case, however, the generator is
determined implicitly and we cannot even say whether its domain contains
$\tilde\scrC$.
\end{remark}

As an illustration, in the examples below we discuss 3 variants of choice of
the rate function. Let us introduce some notation. Let
\begin{equation}\label{4I}
\varphi(\gamma,x,y):=\frac{\mu(\sigma_{x,y}(d\gamma))}{\mu(d\gamma)}\,, \quad
\gamma\in\Gamma,\ (x,y)\in\tilde\X^2,
\end{equation}
stand for the Radon--Nikod\'ym derivative. Note that
\begin{equation}\label{4K}
\varphi(\sigma_{x,y}(\gamma),x,y)=(\varphi(\gamma, x,y))^{-1}.
\end{equation}
Next, observe that the symmetry condition \eqref{4B} can be restated in the
following form:
\begin{equation}\label{4J}
c(\gamma,x,y)=\varphi^{\frac12}(\gamma,x,y)a(\gamma,x,y) \quad \mbox{\rm
with}\quad a(\gamma,x,y)=a(\sigma_{x,y}(\gamma),x,y).
\end{equation}
Finally, fix an arbitrary  function $u(x,y)$ on $\tilde\X^2$ such that
$u(x,y)=u(y,x)\ge0$ and, for any fixed $x$, $\sum_y u(x,y)<\infty$. For
instance, if $\X$ is the vertex set of a locally finite graph, one may suppose
that $u(x,y)=0$ unless $\{x,y\}$ is an edge.

\begin{example}\label{ex1}
Set
$$
c(\gamma,x,y)=u(x,y)\min(\varphi(\gamma,x,y),1)
$$
(compare with the well-known Metropolis dynamics). Equivalently,
$$
a(\gamma,x,y)=u(x,y)\min\left(\varphi^{\frac12}(\gamma,x,y),\,
\varphi^{-\frac12}(\gamma,x,y)\right),
$$
which satisfies the required symmetry property by virtue of \eqref{4K}. For any
$(x,y)$, the $L^2$-norm of the function $c(\,\cdot\,,x,y)$ is less than or equal to
$u(x,y)$. Consequently, the assumption on $u(x,y)$ guarantees the fulfilment of
the ``$L^2$-condition'' \eqref{4C}.
\end{example}

\begin{example}\label{ex2}
Set
$$
c(\gamma,x,y)=u(x,y)\varphi^{\frac12}(\gamma,x,y),
$$
which means $a(\gamma,x,y)=u(x,y)$. Since the $L^2$-norm of the function
$\varphi^{\frac12}(\,\cdot\,,x,y)$ equals 1, the $L^2$-norm of
$c(\,\cdot\,,x,y)$ equals $u(x,y)$, which again implies the
``$L^2$-condition'' \eqref{4C}.
\end{example}

\begin{example}\label{ex3} Set
$$
c(\gamma,x,y)=u(x,y)(\varphi(\gamma,x,y)+1),
$$
which is equivalent to
$$
a(\gamma,x,y)=u(x,y)\left(\varphi^{\frac12}(\gamma,x,y)+\varphi^{-\frac12}(\gamma,x,y)\right)
$$
(compare with the Glauber dynamics discussed in \cite{ShYoo}). In this case we
cannot dispose of the ``$L^2$-condition'' \eqref{4C};
 we cannot even claim that the function $\varphi(\,\cdot\,,
x,y)$ certainly belongs to $L^2(\Gamma,\mu)$. Instead of this, we observe that
the latter function has $L^1$-norm 1, which implies the fulfillment of the
``$L^1$-condition'' \eqref{4F}. This weaker condition still makes it possible
to apply Theorem \ref{4D}, but gives a less precise description of the generator
of the process.

However, in the concrete case when $\mu$ is one of the gamma kernel measures
$\mu_{z,z'}$, it turns out that the function $\varphi(\,\cdot\,, x,y)$ does
belong to $L^2(\Gamma,\mu)$.
Thus, we can satisfy the
``$L^2$-condition'' \eqref{4C}
 provided that $u(x,y)$ satisfies some additional
assumptions (for instance, it suffices to require that for any fixed $x$,
$u(x,y)$ vanishes for all but finitely many $y$'s). The fact that
$\varphi(\,\cdot\,, x,y)$ is square integrable follows from the results of
\cite{Olsh2}: There it is proved that the Radon--Nikod\'ym derivative
$\varphi(\gamma,x,y)$ belongs to the algebra of functions on $\Gamma$ spanned
by the so-called multiplicative functionals; on the other hand, each such
functional is integrable with respect to $\mu_{z,z'}$, hence any element of the
algebra is integrable, which implies that $(\varphi(\,\cdot\,, x,y))^2$ is
integrable.
\end{example}

\begin{remark}
As seen from the above examples, there is quite a lot of flexibility about the
choice of the rate function. Of course, the rate function should be specified
depending on a concrete problem. Finally, note that in the lattice spin models
of statistical mechanics, due to short-range interaction of spins, the function
$\varphi(\gamma,x,y)$ usually takes a simple form and depends only on a small finite
part of the whole spin configuration $\gamma$. For the gamma kernel measures,
the structure of $\varphi(\gamma,x,y)$ is a much more sophisticated.
\end{remark}

\medskip

{\bf Acknowledgements}. We would like to thank Yuri Kondratiev for useful
discussions.  EL was partially supported by   the project SFB 701 (Bielefeld
University). GO was partially supported by a grant from Simons Foundation
(Simons--IUM Fellowship), the RFBR-CNRS grant 10-01-93114, and the project SFB
701 of Bielefeld University.

\end{document}